\documentclass[review,hidelinks,onefignum,onetabnum]{siamart220329}

\usepackage{lipsum}
\usepackage{amsfonts}
\usepackage{graphicx}
\usepackage{epstopdf}
\usepackage{algorithmic}
\ifpdf
  \DeclareGraphicsExtensions{.eps,.pdf,.png,.jpg}
\else
  \DeclareGraphicsExtensions{.eps}
\fi

\newsiamremark{remark}{Remark}
\newsiamremark{hypothesis}{Hypothesis}
\crefname{hypothesis}{Hypothesis}{Hypotheses}
\newsiamthm{claim}{Claim}

\headers{Greedy Algorithms for Rational Approximation}{J.H. Adler, X. Hu, X. Wang, and Z. Xue}

\title{Improving Greedy Algorithms for Rational Approximation\thanks{Submitted to the editors DATE.
\funding{This work was funded by NSF DMS-2208267}}}

\author{James H. Adler \thanks{Department of Mathematics, Tufts University, Medford, MA 02155, USA
   (\email{James.Adler@tufts.edu}, \email{Xiaozhe.Hu@tufts.edu},
   \email{Zhongqin.Xue@tufts.edu}).}
\and Xiaozhe Hu\footnotemark[2]
\and Xue Wang\thanks{School of Mathematics, Shandong University, Shandong 250100, P. R. China 
    (\email{wangxsdu@mail.sdu.edu.cn}).}
  \and Zhongqin Xue\footnotemark[2]
}

\usepackage{amsopn}

\ifpdf
\hypersetup{
  pdftitle={Improving Greedy Algorithms for Rational Approximation},
  pdfauthor={}
}
\fi

\usepackage{lipsum}
\usepackage{amsfonts}
\usepackage{graphicx}
\usepackage{epstopdf}
\usepackage{algorithmic}
\usepackage{multirow}
\Crefname{ALC@unique}{Line}{Lines} 
\usepackage{bm}
\ifpdf%
  \DeclareGraphicsExtensions{.eps,.pdf,.png,.jpg}
\else
  \DeclareGraphicsExtensions{.eps}
\fi
\usepackage{amsopn}
\usepackage{booktabs}

\begin{document}

\maketitle

\begin{abstract}
When developing robust preconditioners for multiphysics problems, fractional functions of the Laplace operator often arise and need to be inverted. Rational approximation in the uniform norm can be used to convert inverting those fractional operators into inverting a series of shifted Laplace operators. Care must be taken in the approximation so that the shifted Laplace operators remain symmetric positive definite, making them better conditioned. 
In this work, we study two greedy algorithms for finding rational approximations to such fractional operators.
The first algorithm improves the orthogonal greedy algorithm discussed in [Li et al., SISC, 2024]  by adding one minimization step in the uniform norm to the procedure. The second approach employs the weak Chebyshev greedy algorithm in the uniform norm.
Both methods yield non-increasing error.
Numerical results confirm the effectiveness of our proposed algorithms, which are also flexible and applicable to other approximation problems. 
Moreover, with effective rational approximations to the fractional operator, the resulting algorithms show good performance in preconditioning a Darcy-Stokes coupled problem.

\end{abstract}



\section{Introduction}
Multiphysics problems usually enforce coupling interface conditions in weak form by employing Lagrange multipliers \cite{holter2020robust,layton2002coupling}.
Efficient and robust preconditioners for such systems sometimes involve inverting fractional and/or weighted Laplace operators on the interface \cite{boon2022parameter,boon2022robust,budivsa2022rational}. For instance, for the robust preconditioner derived in \cite{holter2020robust} for the coupled Darcy-Stokes problem, there is one diagonal block of the form
$$\mathcal{S}=\mu^{-1}\left(-\Delta+\mathcal{I}\right)^{-1 / 2}+K \mu^{-1}\left(-\Delta+\mathcal{I}\right)^{1 / 2},$$
where $\Delta$ is the Laplace operator, $\mathcal{I}$ is the identity operator, and $K$ and $\mu$ are physical parameters in the model.
To avoid inverting this block directly and to improve computational efficiency and reduce storage requirements, one approach is to employ rational approximation. This converts inverting the fractional operator into inverting a series of shifted Laplace operators. 

Rational approximation can represent  complex functions in a compact way, especially for those with singularities, and helps in extracting information about a function outside its region. 
Commonly used rational approximation algorithms include the classical Remez algorithm \cite{petrushev2011rational}, the Pad{\'e} approximation \cite{graves1981pade}, the Best Uniform Rational Approximation (BURA) \cite{stahl2003best}, and the barycentric rational interpolation \cite{berrut2005recent}, to name a few. Recently, an orthogonal greedy algorithm (OGA) was applied in \cite{li2024reduced} to approximate $f(z)=z^{-\alpha}$, $\alpha \in(0,1)$ in the $L^2$ norm. 
When it is applied to the fractional diffusion problem, it converts the fractional diffusion operator into a series of shifted Laplace operators, which by construction remain symmetric positive definite (SPD).  Moreover, \cite{hofreither2020unified} shows that the associated $L^2$ error estimate of the rational approximation method for the fractional diffusion problem relates to the error in the uniform norm.  Motivated by this observation,  we aim to approximate general fractional operators directly in the uniform norm by means of greedy algorithms, while ensuring that the proposed methods always yield SPD operators.  

To this end, we propose two improvements on the OGA presented in \cite{li2024reduced}. First, we modify the OGA by including one extra step to minimize the error in the uniform norm. 
Second, we implement a weak Chebyshev greedy algorithm (WCGA) using the uniform norm. For both algorithms, it is shown that the uniform error is monotonically non-increasing.
The effectiveness and flexibility of the proposed greedy algorithms for rational approximation are demonstrated via some basic tests.
We further apply the proposed algorithms to construct preconditioners for a Darcy-Stokes problem.

The structure of this paper is as follows. In  Section \ref{sec:overview}, we provide a brief review of rational approximation and the general OGA and WCGA.
In Section \ref{sec:main}, we describe the improved OGA and the WCGA for approximating with rational functions in the uniform norm.
Numerical experiments are presented in Section \ref{sec:num} to demonstrate the  effectiveness and feasibility of the proposed algorithms, and concluding remarks are given in Section \ref{sec:conc}.


\section{Preliminaries}\label{sec:overview} In this section we recall the concept of rational approximation as well as two greedy algorithms: OGA and WCGA. 
\subsection{Rational approximation} To start, we consider a general case of a fractional function of a SPD operator $\mathcal{A}$, denoted as $f(\mathcal{A})$. The key idea is to find a rational function $R(z)$ to approximate $f(z)$,
$$f(z)\approx R(z):=\frac{P_m(z)}{Q_n(z)},$$
where $P_m(z)$, $Q_n(z)$ are polynomials of degree $m$ and $n$, respectively. If $m\leq n$, we convert the rational function $R(z)$ into its partial fraction decomposition, i.e.,
\begin{align}\label{partial}
    f(z)\approx R(z)=c_0+\sum_{j=1}^n \frac{c_j}{z-p_j},
\end{align}
where $c_j$ and $p_j$ denote the residues and poles, respectively. Then, for given $\bm{w}$, the rational approximation is used to calculate $f(\mathcal{A})\bm{w}$ as
$$f(\mathcal{A})\bm{w}=c_0\bm{w}+\sum_{j=1}^n c_i\left(\mathcal{A}-p_j \mathcal{I}\right)^{-1}\bm{w}.$$
Particularly, if $\mathcal{A}=-\Delta$, the problem reduces to solving a series of shifted Laplacian problems, i.e.,  $\bm{u}=\left(\mathcal{A}-p_j \mathcal{I}\right)^{-1}\bm{w}$. Having negative poles ensures that the shifted operator remains SPD, yielding better conditioned matrices, especially for smaller values of $p_j$. Thus, in all our proposed methods, we try to ensure that the rational approximation yields negative poles.

There have been several applications in the literature of utilizing rational approximation for efficiently solving fractional diffusion problems of the form, $\mathcal{A}^\alpha \bm{u}= \bm{w}$, $\alpha \in(0,1)$.
One approach is BURA, which is based on a modified Remez algorithm \cite{harizanov2020analysis,harizanov2018optimal}. To avoid the blow up at the origin, the BURA method considers the best rational approximation, $\hat{R}(z)$, for the function $z^{1-\alpha}$, and then takes $R(z)=\hat{R}(z)/z$ as the rational approximation to $z^{-\alpha}$. Bonito and Pasciak \cite{bonito2015numerical} investigate rational approximations based on the integral representation 
of $z^{-\alpha}$. 
Applying these rational approximation methods to more general functions is inherently complex and not straightforward, and there is no guarantee that the poles are negative.
Recently, \cite{nakatsukasa2018aaa} proposes an Adaptive Antoulas–Anderson (AAA) algorithm for general $f(z)$, which employs greedy selected support points and makes use of a barycentric form to represent rational functions. However, in certain situations, the AAA algorithm produces positive or pairs of complex conjugate poles \cite{budivsa2022rational}. We refer readers to \cite{hofreither2020unified} and the references therein for more details and an overview of the existing numerical methods for rational approximation. 

\subsection{Greedy algorithms} 
Various types of greedy algorithms have been  studied and applied in combinatorial optimization \cite{cerrone2017carousel,gorski2012greedy}, machine learning \cite{lin2013learning,sahin2023greedy}, and graph theory \cite{bal2018greedy}.
Greedy algorithms aim to make the ``best" decision at each step of the process, possibly leading to a loss of global optimality. 
In the current paper, we explore two types of greedy algorithms: the OGA in Hilbert spaces and the WCGA in Banach spaces \cite{temlyakov2011greedy}. 

\subsubsection{Orthogonal greedy algorithm} Let $H$ be a Hilbert space equipped with an inner product $(\cdot, \cdot)$ and norm $\|\cdot\|$, and let $f(z)\in H$ be the target function we aim to approximate. Let $\mathcal{D}=\{g\}_{g \in \mathcal{D}}$ be a given dictionary for the greedy algorithm. 
Define $\varphi_k(z)$ as the approximation and $r_k(z):=f(z)-\varphi_k(z)$ as the residual at the $k$-th iteration. Initially, set $\varphi_0=0$, $r_0=f$, then the OGA is defined in an inductive manner in three steps, see Algorithm \ref{OGA}.

\begin{algorithm}[h!]
\caption{Orthogonal Greedy Algorithm}
\begin{algorithmic}\label{OGA}
\STATE{Input: $\mathcal{D}$, target function $f$, and an integer $n$}
\STATE{Initialize: Set $\varphi_0=0$ and $r_0=f$}
\FOR{$k=1: n$} 
\STATE{Greedy step: $g_k=\underset{g \in \mathcal{D}}{\operatorname{argmax}} \left( r_{k-1},g \right)$.}
\STATE{Orthogonal projection: $\varphi_{k}:=P_{g_1, \cdots, g_{k}} f=\sum_{i=1}^n c_ig_i.$}
\STATE{Set $r_k:=f-\varphi_k$.}
\ENDFOR  
\end{algorithmic} \end{algorithm}
\noindent First, we perform a greedy step to find an element $g_{k} \in \mathcal{D}$ that satisfies 
$$g_k=\underset{g \in \mathcal{D}}{\operatorname{argmax}} \left( r_{k-1},g\right).$$
Next, we take an orthogonal projection of $f$ onto the space spanned by $g_1, g_2, \ldots, g_k$, $P_{g_1, \cdots, g_k} f$ to obtain the next iteration of the approximation $\varphi_k$.
Finally, the residual is updated accordingly. 
If we intend to employ OGA to get the rational approximation for $f$, it is sufficient to choose an appropriate dictionary. Then, $\varphi_k$ will align with the desired rational approximation $R(z)$.




\subsubsection{Weak Chebyshev greedy algorithm} The WCGA, originally developed and analyzed in  \cite{temlyakov2001greedy}, is a generalization of the weak orthogonal greedy algorithm \cite{temlyakov2000weak} adapted for Banach spaces.
It allows a certain degree of deviation from optimal choice and more flexibility in constructing a weak greedy approximation.
A detailed discussion of the convergence and the rate of convergence for WCGA is presented in \cite{temlyakov2000weak}.  
We consider a Banach space $X$ with a norm $\|\cdot\|_{X}$. For $u$ in $X$ and $v$ in the dual space $X^\prime$, the duality pairing $\langle v, u\rangle$ is defined as $\langle v, u\rangle=v(u).$
For any nonzero element, $u\in X$, we denote the norming functional of $u$ by $F_u$, which satisfies
\begin{align}\label{norming}
    \left\|F_u\right\|_X=1,\quad \langle F_u , u \rangle=\|u\|_X.
\end{align}
Here, the norming functional is not unique, and its existence is guaranteed by the Hahn-Banach theorem.
Let $f\in X$ be the target function and $\mathcal{D}=\{g\}_{g \in \mathcal{D}}$ be the given dictionary. 
Define $\varphi_k$ as the approximation and $r_k:=f-\varphi_k$ as the residual at the $k$-th iteration. 
We set $\varphi_0=0$ and $r_0:=f$. Let 
$\left\{t_k\right\}_{k=1}^{\infty} \subset(0,1]$ be a given sequence, which satisfies certain properties (see \cite{temlyakov2011greedy} for more details). The inductive definition of the WCGA consists of three step, see Algorithm \ref{WCGAoriginal}.
\begin{algorithm}[h!]
\caption{Weak Chebyshev Greedy Algorithm}
\begin{algorithmic}\label{WCGAoriginal}
\STATE{Input: $\mathcal{D}$, target function $f$, sequence $\left\{t_k\right\}_{k=1}^{\infty}$}, and an integer $n$
\STATE{Initialize: Set $\varphi_0=0$  and $r_0=f$}
\FOR{$k=1: n$} 
\STATE{Greedy step: Find $g_k\in \mathcal{D}$ such that
$\left\langle F_{r_{k-1}},g_k \right\rangle \geq t_k\sup_{g \in \mathcal{D}}\left\langle F_{r_{k-1}},g\right\rangle$.}
\STATE{Best approximation: Find $\varphi_k \in X_k:=\operatorname{span}\left\{g_j\right\}_{j=1}^k$ such that}
$$
\varphi_k=\underset{\varphi \in X_k}{\operatorname{argmin}}\| f - \varphi \|_X.$$
\STATE{Set $r_k:=f-\varphi_k$.}
\ENDFOR  
\end{algorithmic} \end{algorithm}

\noindent First, in the greedy step, we need to find an element $g_{k} \in \mathcal{D}$ satisfying the condition
$$\left\langle F_{r_{k-1}},g_k \right\rangle \geq t_k\sup_{g \in \mathcal{D}}\left\langle F_{r_{k-1}},g\right\rangle.$$
Here, the sequence $\left\{t_n\right\}_{n=1}^{\infty}$ acts as a slack term that provides the algorithm more flexibility at each greedy step.  
Next, the new approximation $\varphi_k$ is the best approximation for $f$ from the space spanned by $g_1, \cdots, g_k$. In the final step, we update the
residual.




\section{Improved Greedy Algorithms}\label{sec:main}
To efficiently solve fractional elliptic problems of the form $\mathcal{A}^\alpha u= b$, $\alpha \in(0,1)$, where $\mathcal{A}$ is SPD, Li et al. \cite{li2024reduced} employed OGA to approximate $f(z)=z^{-\alpha}$ on $[\epsilon,1]$ in the $L^2$ norm. 
With the fact that the poles for the best uniform approximation of the given $f(z)$ are all simple and negative \cite{stahl2003best}, the authors construct the following dictionary guaranteeing that the OGA always generates negative poles,
$$
\mathcal{D}=\mathcal{D}(P)=\left\{g(z):=\left(\frac{1}{\varepsilon-p}-\frac{1}{1-p}\right)^{-\frac{1}{2}} \frac{1}{z-p}\right\}_{p \in P},
$$
where $P$ is reliant on the spectrum of $\mathcal{A}$. For details on error estimates in the $L^2$ norm see \cite{Barron2008greedy,devore1996some,li2023entropy}. Note that the $L^2$ error of the rational approximation method for the fractional diffusion problem relies on the uniform error of the rational approximation in the following way,
\begin{align*}
  \left\|f(\mathcal{A}) \bm{b}-R\left(\mathcal{A}\right) \bm{b}\right\|_{L^2} \leq\max _{z \in\left[\lambda_{\min }, \lambda_{\max }\right]}\left|f(z)-R(z)\right|\|\bm{b}\|_{L^2},  
\end{align*}
where $R(\cdot)$ is the rational approxiamtion for $f$, and $\lambda_{\min }$ and $\lambda_{\max }$ are the minimum and maximum eigenvalues of $\mathcal{A}$ \cite{hofreither2020unified}.  This inspires us to develop greedy algorithms for the rational approximation directly in the uniform norm. 

In the present work, our discussion extends beyond functions of the form $z^{-\alpha}$. We consider a generalized function $f(\mathcal{A})$, where $\mathcal{A}$ is SPD. Thus, the problem becomes applying OGA and WCGA to find good rational approximations for a continuous function $f(z)$ on $[a, b]$ with $a\ge0$. 
The associated uniform norm of \( f \) is defined as $ 
\|f\|_{\infty} = \sup_{z \in [a, b]} |f(z)|.$

\subsection{Algorithm 1: Improved Orthogonal Greedy Algorithm}
Let \( u, v \in L^2(\Omega) \), where \(\Omega \subset \mathbb{R}\). The \(L^2\) inner product is defined as $\left( v, u \right) = \int_{\Omega} u(z) v(z) \, dz.$
Then, the \(L^2\) norm of a function \( u \in L^2(\Omega) \) is defined as $
\|f\| = \sqrt{\left( u, u \right)}.$ To get the partial fraction decomposition \eqref{partial} in order to approximate $f(z)$ on $[a, b]$, we take into account the following normlized dictionary $\mathcal{D}$ which consists of a class of functions $g(z)$ with $\|g\|=1$,
$$
\mathcal{D}(P)=\left\{g(z):=\left(\frac{1}{a-p}-\frac{1}{b-p}\right)^{-\frac{1}{2}} \frac{1}{z-p}\right\}_{p \in P},
$$
where $P:=[P^{\text{left}},P^{\text{right}}]$ depends on the spectrum of $\mathcal{A}$. 
See Algorithm \ref{improvedOGA} for the details of the improved OGA. 

\begin{algorithm}[h!]
\caption{Improved Orthogonal Greedy Algorithm}
\begin{algorithmic}\label{improvedOGA}
\STATE{Input: $a$, $b$, $\mathcal{D}(P)$, target function $f$, and an integer $n$}
\STATE{Initialize: Set $\varphi_0=0$ and $r_0=f$}
\FOR{$k=1: n$} 
\STATE{Greedy step: $g_k=\underset{g \in \mathcal{D}(P)}{\operatorname{argmax}} \left( r_{k-1},g \right)$.}
\STATE{Orthogonal projection: $\varphi_{k}:=P_{g_1, \cdots, g_{k}} f=\sum_{i=1}^k c_ig_i.$}
\STATE{Set $r_k:=f-\varphi_k$.}
\ENDFOR  
\STATE{Best approximation in the uniform norm: 
Using $\{c_i\}_{i=1}^n$ as a initial guess, find $\varphi_n^* \in H_n:=\operatorname{span}\left\{g_j\right\}_{j=1}^n$ such that}
$$
\varphi_n^*=\underset{\varphi \in H_n}{\operatorname{argmin}}\| f - \varphi \|_{\infty}.$$
\end{algorithmic} \end{algorithm}
\begin{remark}
    At the greedy step, the authors of \cite{li2024reduced} determine the ``optimal" value of the pole by performing an exhaustive search over a discretized interval for $P$. In our algorithm, a particle swarm optimization \cite{wang2018particle} is implemented to look for a pole more accurately, thus providing a better approximation.
\end{remark}

The main novelty of Algorithm \ref{improvedOGA}  lies in incorporating a minimization step after taking an orthogonal projection onto the space spanned by the candidates selected from the original OGA. In addition, we use the coefficient $c_i$, computed through the orthogonal projection, as the initial guess for the best approximation to achieve faster convergence. The extra minimization
step ensures the corresponding uniform error always gets smaller compared with the one from the OGA. Consequently,  $\varphi_n^*$ is a better rational approximation for $f(z)$ compared with the one obtained by the original OGA.

\begin{theorem}\label{OGAth}
The uniform error for the rational approximation obtained by the improved OGA satisfies
$$\| f - \varphi_n^* \|_{\infty}\le\| f - \varphi_n \|_{\infty} \text{ and } \| f - \varphi_n^* \|_{\infty}\le\| f - \varphi_m^* \|_{\infty}\text{ for $n\ge m $},$$ 
where $\varphi_n$ is the rational approxiation for $f$ using the OGA with input integer $n$, and $\varphi_n^*$ is the one computed using the improved OGA.
\end{theorem}
\begin{proof}
    Suppose the approximation, $\varphi_n$, by the OGA with input integer $n$ is of the following form
    $$\varphi_n=\sum_{i=1}^n c_ig_i.$$
    Since the best approximation $\varphi_n^*$ of $f$ for the improved OGA is obtained by using $\{c_i\}_{i=1}^n$ as an initial guess, we have $\| f -\varphi_n^* \|_{\infty}\le\| f - \varphi_n \|_{\infty}$.
    
    Similarly, if the best approximation $\varphi_{n-1}^*$ for the improved OGA with input interger $n-1$ is expressed by
    $$\varphi_{n-1}^*=\sum_{i=1}^{n-1} \Tilde{c}_i^{(n-1)}g_i,$$
    then, by setting $d_i^{(n)}=\Tilde{c}_i^{(n-1)}$ for $i=1,2,\cdots,n-1$ and $d_k^{(n)}=0$, we arrive at 
    $$\| f - \varphi_{n}^* \|_{\infty}\le\| f - \sum_{i=1}^{n} d_i^{(n)}g_i\|_{\infty}\le
    \| f - \varphi_{n-1}^* \|_{\infty}.$$
    Then, the desired results are obtained recursively.
\end{proof}

\subsection{Algorithm 2: Weak Chebyshev Greedy Algorithm in the Uniform Norm}  
Given the weaker assumptions of the WCGA, we are able to find a more practical implementation.  Here, we directly apply the WCGA in Algorithm 2.2 with the Banach space, $X = C(\Omega)$, where $\Omega\subset\mathbb{R}$ is a compact set.
Denote the dual space of \( C(\Omega) \) as \( C^\prime(\Omega) \).

Let $\left\{t_k\right\}_{k=1}^{\infty}$ with $0< t_k \leq 1$ be a given  sequence. We define the dictionary,
$$
\mathcal{D}(P)=\left\{g(z):=\frac{1}{z-p}\right\}_{p \in P},
$$
where $P:=[P^{\text{left}},P^{\text{right}}]$ depends on the spectrum of the operator $\mathcal{A}$. 
For any nonzero element, $v\in \mathcal{D}(P)$, the norming functional of $u$, denoted by $F_u$ is defined as: 
$$
\langle F_u,v\rangle = \mathrm{\text{sign}}(u(z^*))v(z^*),\quad\text{with }z^*=\arg \max_{z\in[a,b]} |u(z)|.
$$
One can easily verify that $F_u$ satisfies the conditions in \eqref{norming}, that is
$$
\left\|F_u\right\|_{\infty}=\sup_{v \in \mathcal{D}}\frac{\mathrm{\text{sign}}(u(z^*))v(z^*)}{\left\|v\right\|_{\infty}}=1,
$$
and
$$
\langle F_u,u\rangle=\mathrm{\text{sign}}(u(z^*))u(z^*)=\|u\|_{\infty}.
$$

With this setup, we turn to the crucial part of the WCGA, implementing the greedy step.  Here, the selection of $g_k$ is the process of choosing the pole
at the $k$-th iteration, denoted by $p_k$,  from $P$.
Note that $\left\langle F_{r_{k-1}},g_k \right\rangle \geq t_k\sup_{g \in \mathcal{D}}\left\langle F_{r_{k-1}},g \right\rangle$ is equivalent to
\begin{align*}
    \mathrm{\text{sign}}(r_{k-1}(z_{k-1}^*))\frac{1}{z_{k-1}^*+p_k}\ge t_{k}\sup_{p \in P}\left[\mathrm{\text{sign}}(r_{k-1}(z_{k-1}^*))\frac{1}{z_{k-1}^*+p_k}\right].
\end{align*}
If $\mathrm{\text{sign}}(r_{k-1}(z_{k-1}^*))<0$, we have
\begin{align*}
    \frac{1}{z_{k-1}^*+p_k}&\le t_{k}\frac{1}{z_{k-1}^*+P^{\text{right}}},\\
    p_k&\ge \frac{z_{k-1}^*+P^{\text{right}}}{t_{k}}-z_{k-1}^*.
\end{align*}
Consequently, we define $\Tilde{P}_k(z_{k-1}^*):=[\frac{z_{k-1}^*+P^{\text{right}}}{t_{k}}-z_{k-1}^*,P^{\text{right}}]$ as the domain of possible values for the $k$-th pole. 
On the other hand, if $\mathrm{\text{sign}}(r_{k-1}(z_{k-1}^*))>0$, we have
\begin{align*}
    \frac{1}{z_{k-1}^*+p_k}&\ge t_{k}\frac{1}{z_{k-1}^*+P^{\text{left}}},\\
    p_k&\le \frac{z_{k-1}^*+P^{\text{left}}}{t_{k}}-z_{k-1}^*.
\end{align*}
Then, the domain becomes $\Tilde{P}_k(z_{k-1}^*):=[P^{\text{left}},\frac{z_{k-1}^*+P^{\text{left}}}{t_{k}}-z_{k-1}^*]$. We are able to choose appropriate poles by employing a restart procedure, thereby resulting in an effective implementation of the WCGA.
Finally, $\varphi_n$ serves as the intended rational approximation $R(z)$ for $f(z)$. The overall implementation is presented in Algorithm \ref{alg:restart}.
\begin{algorithm}[h!]
\caption{Implementation of the WCGA in the uniform norm}
\label{alg:restart}
\begin{algorithmic}
\STATE {Input: $a$, $\mathcal{D}(P)$ with $P:=[P^{\text{left}},P^{\text{right}}]$, target function $f$, sequence $\left\{t_k\right\}_{k=1}^{\infty}$, and integers $n$ and $m$}
\STATE{Initialize: $\varphi_0=0$, $r_0=f$ and $z_0=\arg \max_{z\in[a,b]} |r_0(z)|$}
\FOR{$k=1: n$}
\IF{$\mathrm{\text{sign}}(r_{k-1}(z_{k-1}^*))<0$}
        \STATE  $\Tilde{P}^{\text{left}}_k :=\frac{z_{k-1}^*+P^{\text{right}}}{t_{k}}-z_{k-1}^*$ and $\Tilde{P}^{\text{right}}_k :=P^{\text{right}}$
        \ELSE 
        \STATE  $\Tilde{P}^{\text{left}}_k :=P^{\text{left}}$ and $\Tilde{P}^{\text{right}}_k :=\frac{z_{k-1}^*+P^{\text{left}}}{t_{k}}-z_{k-1}^*$
    \ENDIF
\STATE Discretize $\Tilde{P}_k:=[\Tilde{P}^{\text{left}}_k,\Tilde{P}^{\text{right}}_k]$ into $m+1$ points: $\Tilde{P}^i_k = \Tilde{P}^{\text{left}}_k + \frac{i}{m} (\Tilde{P}^{\text{right}}_k - \Tilde{P}^{\text{left}}_k)$ for $i = 0, 1, \ldots, m$
\FOR{$i = 0$ to $m$}
    \STATE{Find $\varphi_k \in X_k:=\operatorname{span}\left\{g_1,\cdots,g_{k-1},g(\Tilde{P}^i_k))\right\}$ such that}
   $$
   \varphi_k=\underset{\varphi \in X_k}{\operatorname{argmin}}\| f - \varphi \|_{\infty}.$$
    \IF{$\| f - \varphi_k\|_{\infty}<\| f - \varphi_{k-1} \|_{\infty}$}
        \STATE Take $g_k=g(\Tilde{P}^i_k)$
        \STATE \textbf{break}
    \ENDIF
\ENDFOR
\STATE{Set $r_k:=f-\varphi_k$ and $z_k^*=\arg \max_{z\in[a,b]} |r_k(z)|$}
\ENDFOR
\end{algorithmic}
\end{algorithm}

\begin{theorem}
    The uniform error for the rational approximation obtained by the WCGA in Algorithm \ref{alg:restart} satisfies 
    \begin{align}\label{WCGAinq}
        \| f - \varphi_k \|_{\infty}\le\| f - \varphi_{k-1} \|_{\infty} \text{ for $k=2,\cdots,n$},
    \end{align}
    where $\varphi_k$ is the approxiation of $f$ at the $k$-iteration using the WCGA.
\end{theorem}
\begin{proof}
    Let $\varphi_{k-1}$ be the approximation computed at the $(k-1)$-th iteration, given by
    $$\varphi_{k-1}=\sum_{j=1}^{k-1} \Tilde{c}_j^{(k-1)}g_j.$$
     If we can find $\varphi_k$ that 
    satisfies the \textit{if} statement at the $k$-th greedy step, then the inequality \eqref{WCGAinq} strictly holds. If the condition is not met for any $1\le i\le m+1$, then we take $g_k=g(\Tilde{P}^{\text{right}}_k)$.
    By Setting $d_j^{(k)}=\Tilde{c}_j^{(k-1)}$ for $j=1,2,\cdots,k-1$ and $d_k^{(k)}=0$, we have the following estimate:
    $$\| f - \varphi_{k}\|_{\infty}\le\| f - \sum_{j=1}^{k} d_j^{(k)}g_j\|_{\infty}\le
    \| f - \varphi_{k-1} \|_{\infty}.$$
\end{proof}
\section{Numerical results}\label{sec:num} In this section, we provide several experiments to illustrate the performance of the algorithms discussed in the previous section. In the implementation, we use the \textit{fmincon} function in Matlab for computing the best approximation in the uniform norm. Since \textit{fmincon} might find any local minima that satisfies the given objectives, we utilize a deflation technique \cite{tarek2022simplifying} by introducing a nonlinear constraint to ensure that the algorithm escapes from local minima and identifies the global minimum. Moreover, we set the sequence in Algorithm \ref{alg:restart} as $t_k=1/\sqrt{k}$, where larger $k$ gives more flexibility in selecting a candidate for the greedy approximation. In addition, we set $m=100$ for the discretization of $\Tilde{P}_k$.

\subsection{Function approximation} To begin, we investigate the effectiveness and flexibility of the proposed algorithms. 

\textbf{Example 1.} Consider the target function $f(z)=z^{-\frac{1}{2}}$ on $[10^{-6},1]$. This problem was considered by Li et al. \cite{li2024reduced} relating to the fractional diffusion problem $\mathcal{A}^{\frac{1}{2}} \bm{u}= \bm{b}$. Here, we compare the performance of the proposed algorithms with the OGA studied in \cite{li2024reduced} to show the effectiveness of the proposed algorithms in the uniform norm.  The function $f(z)=z^{-\frac{1}{2}}$ goes to infinity as $z$ approaches 0, thus we implement the Algorithm \ref{improvedOGA} over the domain $[10^{-8},1]$ by setting $a=10^{-8}$ and $b=1$, but compute the uniform error over the domain $[10^{-6},1]$ to obtain a more precise approximation. 

\begin{table}[htb!]
 \begin{center}
  \caption{Poles and uniform error of the OGA, improved OGA, and WCGA rational approximation of $z^{-0.5}$.}\label{compareOGA1} 
  \vspace*{0.3pt}
  \def\temptablewidth{0.9\textwidth}
  {\rule{\temptablewidth}{0.5pt}}
  \begin{tabular*}{\temptablewidth}{@{\extracolsep{\fill}}ccccccc}
   \multirow{2}{*}{$j$}  & \multicolumn{2}{c}{OGA} & \multicolumn{2}{c}{improved OGA} &\multicolumn{2}{c}{WCGA}\\ \cline{2-3}\cline{4-5}\cline{6-7}
   \multicolumn{1}{c}{} & $p_j$ & Error & $p_j$ & Error & $p_j$ & Error\\ \hline
 1 & -1.2e-04 & 7.2e+02 & -1.2e-04 & 3.5e+02 & -2.5e-09  & 2.1e+02\\
 2 & -1.7e-01 & 7.6e+02 & -1.7e-01 & 3.2e+02 & -2.5e-06  & 8.0e+01\\
 3 & -2.5e-09 & 4.7e+02 & -2.5e-09 & 5.4e+01 & -1.5e-06  & 5.5e+01\\
 4 & -1.6e-06 & 1.9e+01 & -1.6e-06 & 8.2e+00 & -1.5e-06  & 4.2e+01\\
 5 & -4.1e-03 & 3.9e+01 & -4.1e-03 & 7.9e+00 & -1.4e-05  & 3.0e+01\\
 6 & -2.5e+01 & 4.3e+01 & -2.5e+01 & 7.9e+00 & -5.8e-04  & 5.5e+00\\
 7 & -1.9e-05 & 3.3e+01 & -1.9e-05 & 9.3e-01 & -1.2e-02  & 2.7e+00\\
 8 & -2.6e-02 & 3.4e+01 & -2.6e-02 & 8.2e-01 & -2.7e-06  & 1.6e+00\\
 9 & -8.0e-07 & 2.3e+01 & -8.0e-07 & 6.6e-01 & -4.2e-03  & 1.0e+00\\
10 & -6.5e-04 & 1.8e+01 & -6.5e-04 & 2.6e-01 & -2.4e-02  & 7.8e-01\\
11 & -7.9e-01 & 1.7e+01 & -7.9e-01 & 2.5e-01 & -8.3e-01  & 6.5e-01\\
12 & -1.3e-07 & 1.3e+00 & -1.3e-07 & 7.7e-02 & -2.0e-05  & 2.7e-01\\
  \end{tabular*}
    {\rule{\temptablewidth}{0.5pt}}
 \end{center}
\end{table}

The poles $p_j$ and the error in successive iterative steps associated with the OGA, the improved OGA, and the WCGA are displayed in Table \ref{compareOGA1}. Note that we only adopt the minimization procedure in the uniform norm at the last iteration for the improved OGA. We observe that for the case with the same number of poles, the improved OGA and the WCGA have smaller uniform error compared to the original OGA. Moreover, as expected, the uniform error for the improved OGA and the WCGA is monotonically decreasing. The overall computation time for the improved OGA is 1208.46 seconds, while for the WCGA it is 13.31 seconds. 

\textbf{Example 2.} Our second test case is motivated by the Darcy–Stokes model \cite{budivsa2022rational}, which we discuss further in Section \ref{app:pre}. 
The target function is given by $f(z)=\left(sz^\alpha+  tz^\beta\right)^{-1}$ with $z\in[10^{-6},1]$.
For this case, the AAA algorithm yields a  small uniform error, however, it sometimes generates positive or pairs of complex conjugate poles.
We set $s=0.1, t=1, \alpha=1/2$, and $\beta=-1/2$ in $f(z)$. Comparisons with the OGA, the improved OGA, and WCGA are provided in Table \ref{compareOGA2}. 
As the table clearly shows, the proposed greedy algorithms guarantee negative poles, despite sacrificing some of the accuracy in the rational approximation. Again, we see that the conclusion reached in Example 1 is still valid. Here, the improved OGA requires a total computation time of 454.03 seconds, whereas the WCGA requires 6.50 seconds.

\begin{table}[htb!]
 \begin{center}
  \caption{Poles and uniform error of OGA, improved OGA, and WCGA rational approximation of $(0.1z^{0.5}+z^{-0.5})^{-1}$.}\label{compareOGA2} 
  \vspace*{0.3pt}
  \def\temptablewidth{0.8\textwidth}
  {\rule{\temptablewidth}{0.5pt}}
  \begin{tabular*}{\temptablewidth}{@{\extracolsep{\fill}}ccccccc}
   \multirow{2}{*}{$j$}  & \multicolumn{2}{c}{OGA} & \multicolumn{2}{c}{improved OGA} &\multicolumn{2}{c}{WCGA}\\ \cline{2-3}\cline{4-5}\cline{6-7}
   \multicolumn{1}{c}{} & $p_j$ & Error & $p_j$ & Error & $p_j$ & Error\\ \hline
 1 & -2.5e+01 & 6.4e-01 & -2.5e+01 & 4.6e-01 & -2.5e-09  & 9.1e-01\\
 2 & -2.8e-02 & 5.7e-01 & -2.8e-02 & 2.2e-01 & -4.1e-01  & 7.0e-01\\
 3 & -1.4e-03 & 1.2e+00 & -1.4e-03 & 2.0e-01 & -7.3e-01  & 1.7e-01\\
 4 & -2.5e-01 & 2.6e-01 & -2.5e-01 & 4.8e-02 & -1.0e+00  & 6.9e-02\\
 5 & -6.2e-05 & 5.5e-01 & -6.2e-05 & 4.8e-02 & -3.1e-02 & 4.5e-02\\
 6 & -6.6e-03 & 4.6e-01 & -6.6e-03 & 3.9e-02 & -1.4e+00  & 3.5e-02\\
 7 & -9.9e-01 & 4.7e-02 & -9.9e-01 & 3.8e-03 & -1.6e+00  & 2.2e-02\\
  \end{tabular*}
  {\rule{\temptablewidth}{0.5pt}}
\end{center}
\end{table}

\textbf{Example 3.} To demonstrate that the proposed algorithms are also applicable for other approximations beyond rational approximation, we utilize a ``proper" dictionary for the function $f(z)=(sz^{\alpha}+tz^{\beta})^{-1}$ on $[10^{-6},1]$. Here we take $s=0.1$, $t=1$, $\alpha=0.4$ and $\beta=0.6$.
The dictionary is constructed as
\begin{align}\label{extdic}
    \mathcal{D}:=\left\{g(z):=z^{-\eta}\right\}_{\eta \in(0,1)}.
\end{align}
This is due to the fact that the graphs of the target function and the elements in the dictionary exhibit a similar pattern characterized by the function value increasing sharply when $z$ is close to 0. 
Hence, we aim to approximate $f(z)$ with linear combinations of $z^{-\eta}$ with different values of $\eta$, i.e., $f(z)\approx\sum_{j=1}^{n}c_jz^{-\eta_j}$.

In Table \ref{extended3}, we record the value of $\eta_j$ and the corresponding coefficients $c_j$ for the improved OGA and the WCGA, respectively. The proposed algorithms show good approximations of $f(z)$, where the improved OGA requires 7 terms and the WCGA needs 13 terms for reducing the residual
by $5\times10^{-2}$. 
Specifically, the uniform errors are $2.5\times 10^{-2}$ and $3.9\times 10^{-2}$ for the improved OGA and WCGA, respectively. This implies that in addition to yielding good rational approximations, the proposed greedy algorithms are also effective for other approximation problems. In addition, given the comparable uniform error, the total computation time for the improved OGA is 1350.01 seconds, while the required time for the WCGA is 196.74 seconds, which is approximately 85\% less than the former.


\begin{table}[ht]
\centering
\caption{Approximation of $(0.1z^{0.4}+z^{0.6})^{-1}$ using improved OGA and WCGA with in dictionary \eqref{extdic}}.\label{extended3} 
\begin{minipage}[t]{0.4\linewidth} 
\centering
  \vspace*{0.3pt}
  \def\temptablewidth{1\textwidth}
  {\rule{\temptablewidth}{0.4pt}}
  \begin{tabular*}{\temptablewidth}{@{\extracolsep{\fill}}ccc}
   \multirow{2}{*}{$j$}  & \multicolumn{2}{r}{improved OGA} \\ \cline{2-3}\cline{2-3}
   \multicolumn{1}{c}{} & $\eta_j$ & $c_j$  \\ \hline
   1 & 1.0e+00 & -1.6e-01 \\
   2 & 8.7e-01 & -5.1e-01 \\
   3 & 6.6e-01 &  1.1e+00 \\
   4 & 2.2e-01 & -8.2e-01 \\
   5 & 9.5e-01 &  4.4e-01 \\
   6 & 1.0e-08 &  2.7e-01 \\
   7 & 4.8e-01 &  1.8e+00 \\
  \end{tabular*}
  {\rule{\temptablewidth}{0.5pt}}
\end{minipage}
\hspace{0.05\textwidth}
\begin{minipage}[t]{0.4\linewidth} 
\centering
 
  \vspace*{0.3pt}
  \def\temptablewidth{1\textwidth}
  {\rule{\temptablewidth}{0.4pt}}
  \begin{tabular*}{\temptablewidth}{@{\extracolsep{\fill}}ccc}
   \multirow{2}{*}{$j$}  & \multicolumn{2}{r}{WCGA} \\ \cline{2-3}\cline{2-3}
   \multicolumn{1}{c}{} & $\eta_j$ & $c_j$ \\ \hline
   1 & 1.0e+00 &  3.2e-02\\
   2 & 9.7e-01 & -8.2e-02\\
   3 & 9.5e-01 & -8.0e-02\\
   4 & 9.4e-01 &  6.9e-02\\
   5 & 8.7e-01 &  1.2e-01\\
   6 & 9.3e-01 &  9.7e-02\\
   7 & 8.4e-01 & -6.4e-02\\
   8 & 7.6e-01 & -2.5e-01\\
   9 & 7.1e-01 &  5.0e-02\\
  10 & 6.7e-01 &  2.0e-01\\
  11 & 8.3e-01 & -1.0e-01\\
  12 & 6.3e-01 &  4.3e-01\\
  13 & 6.1e-01 &  5.7e-01\\
  \end{tabular*}
  {\rule{\temptablewidth}{0.5pt}}
\end{minipage}
\end{table}

\subsection{Preconditioning of the Darcy-Stokes problem}\label{app:pre}
Finally, we investigate the application of rational approximation on preconditioning of the Darcy-Stokes coupled model.
Instead of solving the problem using the preconditioner involving a fractional block directly, we apply the proposed greedy algorithms to approximate the inverse of the fractional operator with rational functions.

The Darcy-Stokes coupled problem models 
slow fluid flow in physical systems consisting of a free flow domain and a rigid porous media. Specifically, the Stokes equation is employed in the free flow domain, while Darcy’s law governs the flow in the porous region. 
Consider the Darcy-Stokes problem on a bounded domain $\Omega :=\Omega_S \bigcup \Omega_D$ with $\Omega_S, \Omega_D\subset \mathbb{R}^d, d=2,3$, and interface $\Gamma :=\partial \Omega_S \cap \partial \Omega_D$:
\begin{subequations}\label{DSpro}
\begin{alignat}{2}
-\nabla \cdot \boldsymbol{\sigma}_S\left(\boldsymbol{u}_S, p_S\right)  =\boldsymbol{f}_S \text {\quad and\quad} \nabla \cdot \boldsymbol{u}_S & =0 & \text { in } \Omega_S, \\
\boldsymbol{u}_S &=0 & \text { on } \partial \Omega_S^E,\\
\boldsymbol{\sigma}_S\left(\boldsymbol{u}_S, p_S\right) \cdot \boldsymbol{n}_S & =0 & \text { on } \partial \Omega_S^N,\\
\boldsymbol{u}_D  =-K\mu^{-1} \nabla p_D \text {\quad and\quad} \nabla \cdot \boldsymbol{u}_D & =f_D & \text { on } \Omega_D, \\
\boldsymbol{u}_D \cdot \boldsymbol{n}_D  &=0 & \text { on }  \partial \Omega_D^E,\\
p_D&=0 & \text { on }  \partial \Omega_D^N,
\label{DSproblem4}
\end{alignat}
\end{subequations}
where $\partial \Omega_S \backslash \Gamma=\partial \Omega_S^E \cup \partial \Omega_S^N,\left|\partial \Omega_S^i\right|>0$,  $\Gamma \cap \partial \Omega_S^E=\emptyset$, $\partial \Omega_D \backslash \Gamma=\partial \Omega_D^E \cup \partial \Omega_D^N,\left|\partial \Omega_D^j\right|>0$,  $\Gamma \cap \partial \Omega_D^N=\emptyset$, $i,j=E,N$. Here, $\sigma_S\left(\bm{u}_S, p_S\right)=2\mu \epsilon(\bm{u}_S)-p_S \mathcal{I}$ represents a stress tensor, with deformation rate tensor $\epsilon(\bm{u}_S):=\frac{1}{2}(\nabla \bm{u}_S+\nabla \bm{u}_S^T)$, $\bm{f}_S$ and $f_D$ denote the external forces, $\bm{u}_S$ and $p_S$ are the velocity and pressure in the fluid region $\Omega_S$, and $\bm{u}_D$ and $p_D$ are the velocity and pressure in the porous media $\Omega_D$. The kinematic viscosity of the fluid is $\mu>0$, and the Darcy permeability is $K>0$. Moreover, $\bm{n}_S$ and $\bm{n}_D$ are the outward unit vectors that are normal to the boundary $\partial \Omega_S$ and $\partial \Omega_D$, respectively.
The coupling of the two subdomains is achieved by imposing three interface conditions on $\Gamma$:
\begin{subequations}\label{con}
\begin{align}
\bm{u}_S \cdot \bm{n}-\bm{u}_D \cdot \bm{n} & =0, \\
\bm{n} \cdot \bm{\sigma}_S\left(\bm{u}_S, p_S\right) \cdot \bm{n}+p_D & =0, \\
\bm{n} \cdot \bm{\sigma}_S\left(\bm{u}_S, p_S\right) \cdot \bm{\tau}+\xi\bm{u}_S \cdot \bm{\tau} & =0,
\end{align}
\end{subequations}
where $\xi = \alpha_{BJS} \mu K^{-1 / 2}$ with the Beavers-Joseph-Saffman (BJS) coefficient $\alpha_{BJS}>0$. Additionally,   $\bm{n}$ is  defined as the unit normal on the interface pointing from the fluid region $\Omega_S$  towards the porous structure $\Omega_D$ and $\bm{\tau}$ represents a unit vector tangent to the interface. 


Let $H^s(\Omega)$ with real number $s \in(-1,1)$ be the Sobolev space defined via the spectral decomposition of the Laplacian. 
To derive the weak formulation associated with the Stokes-Darcy system, we introduce Lagrange multiplier $\lambda\in Q=\Lambda(\Gamma)$ with $\Lambda(\Gamma):=\mu^{-1 / 2} H^{-\frac{1}{2}}(\Gamma) \cap \mu^{-1 / 2}K^{1 / 2} H^{\frac{1}{2}}(\Gamma)$ at the interface.
Let $V_S$, $V_D$, $Q_S$, $Q_D$ be function spaces for the Stokes velocity, the Darcy flux, Stokes pressure, and Darcy pressure, respectively. Then, the variational formulation of the problem reads: Find $\left(\bm{u}_S, p_S, \bm{u}_D, p_D, \lambda\right) \in V_S \times Q_S \times V_D \times Q_D \times \Lambda$ such that 
$$
\left(\begin{array}{ccccc}
-\mu \nabla\cdot\epsilon+\xi T_t^{\prime} T_t & & -\nabla & & T_n^{\prime} \\
& \mu K^{-1} \mathcal{I} & & -\nabla & -T_n^{\prime} \\
 \nabla \cdot & & & &\\

& \nabla \cdot & & & \\
T_n & -T_n & &
\end{array}\right)\left(\begin{array}{c}
\bm{u}_S \\
\bm{u}_D \\
p_S \\
p_D \\
\lambda
\end{array}\right)=\left(\begin{array}{c}
\boldsymbol{f}_S \\
0 \\
0 \\
f_D \\
0
\end{array}\right),
$$
where the operator $ T_n$ and $ T_t $ represent the normal and tangential trace operators on the interface $\Gamma$, respectively.

Recently, a parameter-robust preconditioner has been constructed for solving the Darcy-Stokes problem in \cite{holter2020robust}, defined as
\begin{align}\label{DSpre}
    \mathcal{M}=\left(\begin{array}{lllll}
    -\mu \nabla\cdot\epsilon+\xi T_t^{\prime} T_t & & & & \\
    & \mu K^{-1}(\mathcal{I}-\nabla \nabla \cdot) & & & \\
    & & \mu^{-1} \mathcal{I} & & \\
    & &   &\mu^{-1}K\mathcal{I} & \\
    & & & & \mathcal{S}
    \end{array}\right)^{-1},
\end{align}
where the fractional operator $\mathcal{S}:=\mu^{-1}\left(-\Delta+\mathcal{I}\right)^{-1 / 2}+K \mu^{-1}\left(-\Delta+\mathcal{I}\right)^{1 / 2}$.
Therefore, we would like to apply rational approximation to approximate $\mathcal{S}^{-1}$ so that the inverse of the fractional operator turns into solving a series of shifted Laplacian problems.

In this example, we set the problem domain $\Omega=[0,2] \times[0,1]$ with $\Omega_S=[0,1) \times[0,1], \Omega_D=(1,2] \times[0,1]$, and $\alpha_{BJS}=1$.
The data $\bm{f}_S$ and $f_D$ are chosen so that the exact solution is given by
\begin{align*}
\bm{u}_S&=\left(\begin{array}{c}
(e^x-e)\cos (\pi y) \\
-e^x\sin (\pi y)/\pi
\end{array}\right), &\quad p_S=2e^x\cos (\pi y), \quad \text { in } \Omega_S ,\\
\bm{u}_D&=\left(\begin{array}{c}
-(e^x-e)\cos (\pi y) \\
\pi(e^x-ex) \sin (\pi y)
\end{array}\right), &\quad p_D=(e^x-ex) \cos (\pi y), \quad \text { in } \Omega_D .
\end{align*}
For the spatial discretization, we utilize the lowest order Crouzeix-Raviart space \cite{Rui2009A} as the velocity space. Piecewise constant space is applied for the pressure, and we consider a piecewise linear approximation for the Lagrange multiplier. 
Recalling the expression of the fractional operator $\mathcal{S}$, the target function  is given by 
\begin{align*}
f(z)&=\left(\mu^{-1}z^{-1/2}+  K\mu^{-1}z^{1/2}\right)^{-1},\\
&=\mu\left(z^{-1/2}+  Kz^{1/2}\right)^{-1}, \quad
z\in(0,\lambda_{\max}],
\end{align*}
where $\lambda_{\max}$ is the largest eigenvalue of discretization of the operator $-\Delta+\mathcal{I}$. 
Let $c$ be a constant such that $c\ge \lambda_{\max}$. We rescale the problem by considering 
\begin{align*}
\left(z^{-\frac{1}{2}}+K z^{\frac{1}{2}}\right)^{-1}&=\frac{1}{\gamma_0}\left(\frac{c^{-\frac{1}{2}}}{\gamma_0}\left(\frac{z}{c}\right)^{-\frac{1}{2}}+\frac{Kc^{\frac{1}{2}}}{\gamma_0}\left(\frac{z}{c}\right)^{\frac{1}{2}}\right)^{-1} \\
& =:\frac{1}{\gamma_0} \tilde{f}(\tilde{z}) \quad\text{ with $\tilde{z}:=\frac{z}{c}$, $  \tilde{z}\in(0,1]$,} 
\end{align*}
where $\gamma_0:=\max \{c^{-\frac{1}{2}}, Kc^{\frac{1}{2}}\}$. 
Then, the problem turns into finding a good rational approximation for the rescaled funtion of a scalar variable $\Tilde{f}(\tilde{z})$, $\tilde{z} \in(0,1]$. We set the residual tolerance of the proposed greedy algorithms for rational approximation to be $0.1$, as the rational approximation is used in a preconditioner which does not need to be very precise. The following results validate this choice. 

Table \ref{M1} reports the number of
iterations required by GMRES preconditioned with the operator $\mathcal{M}$. Tables \ref{OGA1}-\ref{WCGA1} list the relative changes in number of iterations by substituting the rational approximation for the fractional operator $S$ via the improved OGA and WCGA, respectively. For example, a value of +3 indicates an increase of 3 iterations compared to the one computed using fractional operator $S$ directly, while a value of -1 indicates a decrease of 1 iteration. 
Compared with inverting
fractional operator $\mathcal{S}$ directly, the approximations employing the improved OGA and WCGA show similar performance which are parameter robust with respect to the viscosity $\mu$, permeability $K$, and mesh size $h$. This demonstrates that our proposed algorithms exhibit good performance in preconditioning of the Darcy-Stokes problem.
Figure \ref{poles1} and Figure \ref{poles2} present the number of poles needed for rational approximation. 
The number varies from 2 to 15 for different cases. The average number of poles required by the OGA and WCGA is comparable for this example, with the OGA at approximately 7.16 and the WCGA at approximately 7.44. Moreover, we observe a trend that as $K$ decreases, the number of poles decreases accordingly, which makes the resulting algorithms appealing for solving complex problems. 

\begin{table}[h]
\renewcommand{\arraystretch}{1.1}
\centering
\begin{center}
\caption{Preconditioned GMRES iterations required for convergence using fractional operator $\mathcal{S}$ directly.}  \label{M1}
\begin{tabular}{c|ccccc|ccccc}
\hline & \multicolumn{5}{c|}{$\mu=1$} 
& \multicolumn{5}{c}{$\mu=10^{-2}$}
\\
\hline${K}\backslash h$ 
& $2^{-2}$ & $2^{-3}$ & $2^{-4}$ & $2^{-5}$ & $2^{-6}$ 
& $2^{-2}$ & $2^{-3}$ & $2^{-4}$ & $2^{-5}$ & $2^{-6}$\\
\hline 
1         &25 &25 &25 &26 &26 &29 &33 &35 &37 &37\\
$10^{-1}$ &25 &27 &28 &28 &28 &31 &33 &35 &37 &38\\
$10^{-2}$ &25 &27 &27 &28 &27 &31 &35 &35 &37 &39\\
$10^{-3}$ &21 &25 &25 &28 &29 &31 &35 &37 &39 &39\\
$10^{-4}$ &19 &19 &21 &22 &25 &29 &32 &35 &37 &38\\
\hline & \multicolumn{5}{c|}{$\mu=10^{-4}$} 
& \multicolumn{5}{c}{$\mu=10^{-6}$}
\\
\hline${K}\backslash h$ 
& $2^{-2}$ & $2^{-3}$ & $2^{-4}$ & $2^{-5}$ & $2^{-6}$
& $2^{-2}$ & $2^{-3}$ & $2^{-4}$ & $2^{-5}$ & $2^{-6}$\\
\hline 
1         &33 &38 &41 &45 &45 &39 &45 &49 &52 &55\\
$10^{-1}$ &35 &39 &41 &45 &47 &39 &47 &49 &54 &56\\
$10^{-2}$ &35 &41 &43 &45 &47 &39 &47 &51 &55 &55\\
$10^{-3}$ &37 &43 &47 &47 &47 &39 &49 &54 &55 &56\\
$10^{-4}$ &39 &45 &47 &49 &51 &43 &51 &56 &59 &59\\
\hline
\end{tabular}
\end{center}
\end{table}

\begin{table}[h]
\renewcommand{\arraystretch}{1.1}
\centering
\begin{center}
\caption{Preconditioned GMRES iterations required for convergence using the improved OGA.}   \label{OGA1}
\begin{tabular}{c|ccccc|ccccc}
\hline & \multicolumn{5}{c|}{$\mu=1$} 
& \multicolumn{5}{c}{$\mu=10^{-2}$}
\\
\hline${K}\backslash h$ 
& $2^{-2}$ & $2^{-3}$ & $2^{-4}$ & $2^{-5}$ & $2^{-6}$ 
& $2^{-2}$ & $2^{-3}$ & $2^{-4}$ & $2^{-5}$ & $2^{-6}$\\
\hline 
1         &+0 &+1 &+1 &+1 &+1 &-1 &-1 &+1 &+2 &+3\\
$10^{-1}$ &+1 &+2 &+1 &+1 &+1 &+0 &+4 &+3 &+3 &+3\\
$10^{-2}$ &-2 &+1 &+1 &-1 &+0 &-2 &+0 &+4 &+2 &+0\\
$10^{-3}$ &+2 &+4 &+0 &+2 &+0 &+1 &+4 &+2 &+0 &+2\\
$10^{-4}$ &+0 &+1 &+0 &+9 &+1 &+1 &+1 &+0 &+4 &+3\\
\hline & \multicolumn{5}{c|}{$\mu=10^{-4}$} 
& \multicolumn{5}{c}{$\mu=10^{-6}$}
\\
\hline${K}\backslash h$ 
& $2^{-2}$ & $2^{-3}$ & $2^{-4}$ & $2^{-5}$ & $2^{-6}$
& $2^{-2}$ & $2^{-3}$ & $2^{-4}$ & $2^{-5}$ & $2^{-6}$\\
\hline 
1         &+1 &-1 &+0 &+3 &+4 &+0 &+0 &+0 &+3 &+2\\
$10^{-1}$ &+0 &+2 &+4 &+3 &+4 &+0 &+2 &+4 &+3 &+3\\
$10^{-2}$ &+0 &+0 &+4 &+3 &+2 &+0 &+2 &+1 &+2 &+2\\
$10^{-3}$ &+0 &+2 &+0 &+2 &+2 &+2 &+4 &+1 &+2 &+2\\
$10^{-4}$ &-2 &+0 &+0 &+4 &+2 &+0 &+0 &+0 &+2 &+4\\
\hline
\end{tabular}
\end{center}
\end{table}

\begin{table}[h]
\renewcommand{\arraystretch}{1.1}
\centering
\begin{center}
\caption{Preconditioned GMRES iterations required for convergence using the WCGA.}   \label{WCGA1}
\begin{tabular}{c|ccccc|ccccc}
\hline & \multicolumn{5}{c|}{$\mu=1$} 
& \multicolumn{5}{c}{$\mu=10^{-2}$}
\\
\hline${K}\backslash h$ 
& $2^{-2}$ & $2^{-3}$ & $2^{-4}$ & $2^{-5}$ & $2^{-6}$ 
& $2^{-2}$ & $2^{-3}$ & $2^{-4}$ & $2^{-5}$ & $2^{-6}$\\
\hline 
1         &+0 &+2 &+1 &+1 &+1 &+1 &-1 &+1 &+3 &+4\\
$10^{-1}$ &+2 &+2 &+1 &+1 &+1 &+0 &+3 &+3 &+3 &+3\\
$10^{-2}$ &+0 &+1 &+1 &+1 &+0 &+0 &+0 &+4 &+4 &+2\\
$10^{-3}$ &+2 &+2 &+1 &+0 &+1 &+1 &+2 &+3 &+0 &+0\\
$10^{-4}$ &+1 &+2 &+0 &+0 &+0 &+3 &+2 &+0 &-2 &+1\\
\hline & \multicolumn{5}{c|}{$\mu=10^{-4}$} 
& \multicolumn{5}{c}{$\mu=10^{-6}$}
\\
\hline${K}\backslash h$ 
& $2^{-2}$ & $2^{-3}$ & $2^{-4}$ & $2^{-5}$ & $2^{-6}$
& $2^{-2}$ & $2^{-3}$ & $2^{-4}$ & $2^{-5}$ & $2^{-6}$\\
\hline 
1         &+1 &+1 &+0 &+3 &+4 &+0 &+2 &+0 &+3 &+4\\
$10^{-1}$ &+0 &+2 &+3 &+4 &+4 &+0 &+0 &+4 &+3 &+3\\
$10^{-2}$ &+0 &+0 &+4 &+4 &+2 &+0 &+1 &+4 &+3 &+3\\
$10^{-3}$ &+0 &+2 &+2 &+0 &+0 &+2 &+2 &+3 &+1 &-1\\
$10^{-4}$ &+0 &+2 &+2 &-2 &+0 &+0 &+2 &+2 &-2 &+2\\
\hline
\end{tabular}
\end{center}
\end{table}

\begin{figure}[h]
        \centering
        \includegraphics[width=0.8\textwidth]{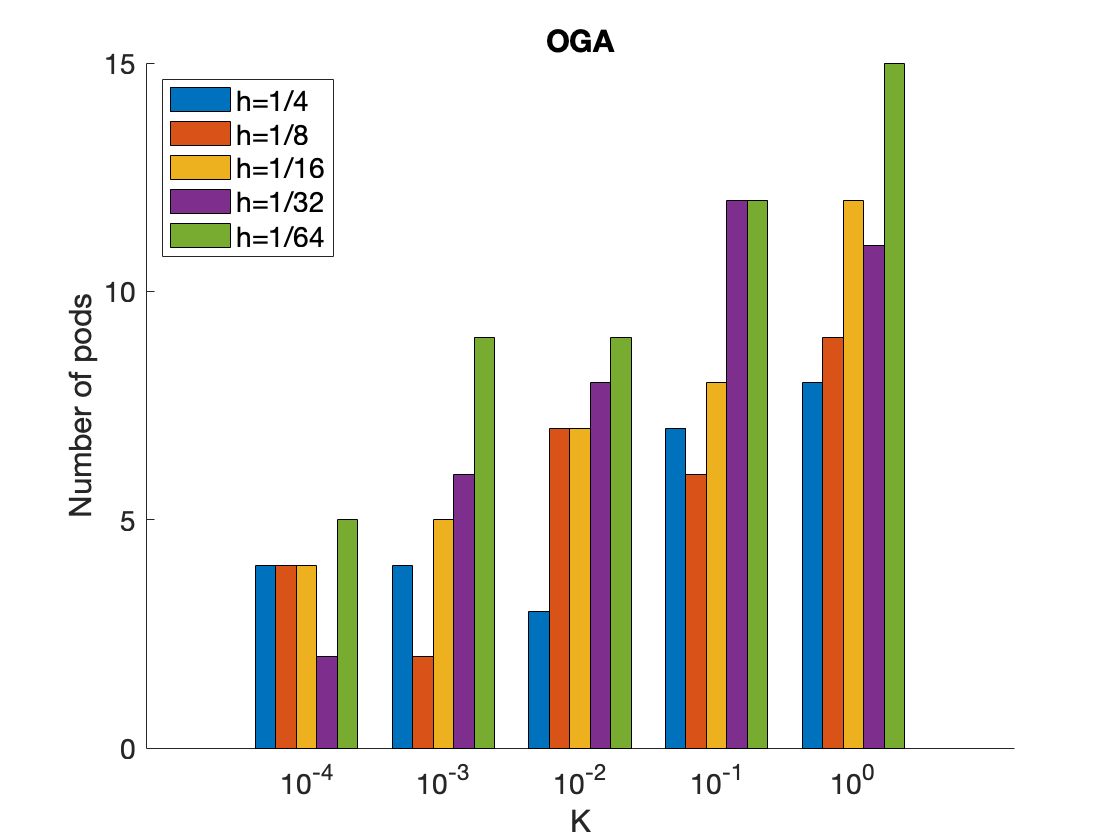}
    \caption{Number of poles needed for the improved OGA applied to the Darcy-Stokes problem.}\label{poles1}
\end{figure}

\begin{figure}[h]
        \centering
        \includegraphics[width=0.8\textwidth]{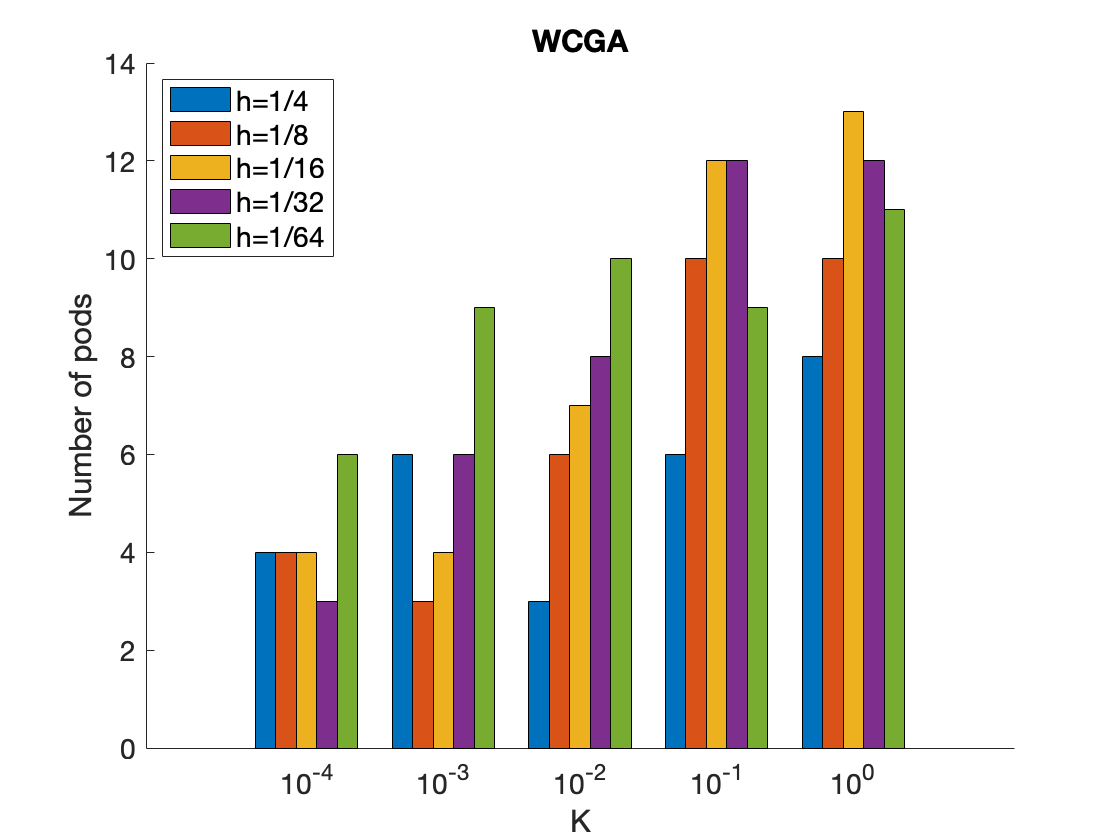}
    \caption{Number of poles needed for the WCGA applied to the Darcy-Stokes problem.}\label{poles2}
\end{figure}
\section{Conclusions}\label{sec:conc}
In this paper, two greedy algorithms are presented for approximating a general  function of a SPD operator. The first method improves the orthogonal greedy algorithm studied in \cite{li2024reduced} by adding a minimization step in the uniform norm after taking an orthogonal projection onto the space spanned by candidates. The second approach is a direct application of the weak Chebyshev greedy algorithm in the uniform norm. Compared with existing algorithms for rational approximation in the uniform norm, the proposed greedy algorithms always guarantee negative poles, keeping the resulting shifted Laplace operators SPD.
Numerical results illustrate that the proposed greedy algorithms are promising for rational approximation and preconditioning of the Darcy-Stokes problem.
Moreover, the proposed algorithms exhibit flexibility and applicability in the sense that their dictionaries can be expanded to consider other approximation problems. 
Also, the WCGA is more practical and efficient for approximation when high accuracy is not a crucial concern.
Our next step is to apply the proposed greedy algorithms to develop preconditioners for other multiphysics problems.
Further investigation into the convergence behavior of the proposed greedy algorithms are necessary for a deeper understanding. 




\bibliographystyle{siamplain}
\bibliography{references}
\end{document}